\documentclass[a4paper]{amsart}
\usepackage{amscd}
\usepackage{amssymb}
\usepackage[T1]{fontenc}
\usepackage[utf8]{inputenc}
\usepackage{hyperref}
\usepackage[arrow,curve,matrix]{xy}
\usepackage{times}
\usepackage{graphicx}
\usepackage{comment}

\DeclareFontFamily{OMS}{rsfs}{\skewchar\font'60}
\DeclareFontShape{OMS}{rsfs}{m}{n}{<-5>rsfs5 <5-7>rsfs7 <7->rsfs10 }{}
\DeclareSymbolFont{rsfs}{OMS}{rsfs}{m}{n}
\DeclareSymbolFontAlphabet{\scr}{rsfs}

 \newcommand{\C}{\mathbb{C}}

\newcommand{\sF}{\scr{F}}
\newcommand{\sG}{\scr{G}}

\newcommand{\sJ}{\scr{J}}

\newcommand{\sO}{\scr{O}}

\newcommand{\bN}{\mathbb{N}}

\DeclareMathOperator{\Jet}{Jet}

\DeclareMathOperator{\Hom}{Hom}

\DeclareMathOperator{\Image}{Image}

\DeclareMathOperator{\rank}{rank}

\DeclareMathOperator{\Spec}{Spec}

\newcounter{thisthm}
\newcommand{\ilabel}[1]{\newcounter{#1}\setcounter{thisthm}{\value{thm}}\setcounter{#1}{\value{enumi}}}
\newcommand{\iref}[1]{(\thesection.\the\value{thisthm}.\the\value{#1})}

\theoremstyle{plain}    
\newtheorem{thm}{Theorem}[section]
\newtheorem{defn}[thm]{Definition}

\newtheorem{assumption}[thm]{Assumption} 
 
\numberwithin{equation}{thm}
\numberwithin{figure}{section}
\theoremstyle{plain}    
\newtheorem{cor}[thm]{Corollary}
\newtheorem{lem}[thm]{Lemma}

\newtheorem{fact}[thm]{Fact}

\theoremstyle{plain}    
\newtheorem{prop}[thm]{Proposition}
\newtheorem{proclaim-special}[thm]{\specialthmname}

\theoremstyle{remark}
\newtheorem{rem}[thm]{Remark}

\newtheorem{warning}[thm]{Warning}

\newtheorem*{claim*}{Claim} 
\newtheorem{notation}[thm]{Notation}
\newtheorem{construction}[thm]{Construction}

\newtheoremstyle{bozont-remark}{3pt}{3pt}%
     {}
     {}
     {\it}
     {.}
     {.5em}
     {\thmname{#1}\thmnumber{ #2}: \thmnote{\sc #3}}
\theoremstyle{bozont-remark}

\def\factor#1.#2.{\left. \raise 2pt\hbox{$#1$} \right/\hskip -2pt\raise
  -2pt\hbox{$#2$}} 
\newlength{\swidth}
\newenvironment{enumerate-p}{
  \begin{enumerate}}
  {\setcounter{equation}{\value{enumi}}\end{enumerate}}

\date{\today}

\author{Stefan Kebekus}
\author{Stavros Kousidis}
\author{Daniel Lohmann}

\thanks{Stefan Kebekus was supported in part by the DFG-Forschergruppe
  ``Classification of Algebraic Surfaces and Compact Complex Manifolds''. Some
  parts of this paper were worked out during a visit of Kebekus to the
  University of Michigan at Ann~Arbor. He would like to thank Mircea Mustaţă
  for the invitation and for a number of discussions. The work on this paper
  was finished while Kebekus visited the 2009 Special Year in Algebraic
  Geometry at the Mathematical Sciences Research Institute, Berkeley. He would
  like to thank the MSRI for support. Daniel Lohmann was supported in part by
  the DFG-Graduiertenkolleg ``Global Structures in Geometry and Analysis''.}
  
\address{Stefan Kebekus, Mathematisches Institut, Albert-Ludwigs-Universit\"at
  Freiburg, Eckerstraße 1, 79104 Freiburg im Breisgau, Germany}
\email{\href{mailto:stefan.kebekus@math.uni-freiburg.de}{stefan.kebekus@math.uni-freiburg.de}}
\urladdr{\href{http://home.mathematik.uni-freiburg.de/kebekus}{http://home.mathematik.uni-freiburg.de/kebekus}}

\address{Stavros Kousidis, Mathematisches Institut der Universit\"at zu K\"oln, Weyertal 86-90, 50931 K\"oln, Germany}
\email{\href{mailto:skousidi@math.uni-koeln.de}{skousidi@math.uni-koeln.de}}
\urladdr{\href{http://www.mi.uni-koeln.de/~skousidi}{http://www.mi.uni-koeln.de/\~{}skousidi}}

\address{Daniel Lohmann, Mathematisches Institut, Albert-Ludwigs-Universit\"at
  Freiburg, Eckerstraße 1, 79104 Freiburg im Breisgau, Germany}
\email{\href{mailto:daniel.lohmann@math.uni-freiburg.de}{daniel.lohmann@math.uni-freiburg.de}}

\title{Deformations along subsheaves}

\begin{document}

\begin{abstract}
  Let $f : Y \to X$ be a morphism of complex manifolds, and assume that $Y$ is
  compact. Let $\sF \subseteq T_X$ be a subsheaf which is closed under the Lie
  bracket.  This short paper contains an elementary and very geometric
  argument to show that all obstructions to deforming $f$ along the sheaf
  $\sF$ lie in $H^1\bigl( Y,\, \sF_Y \bigr)$, where $\sF_Y \subseteq f^*(T_X)$
  is the image of $f^*(\sF)$ under the pull-back of the inclusion map.

  Special cases of this result include Miyaoka's theory of deformation along a
  foliation, Keel-McKernan's logarithmic deformation theory and deformations
  with fixed points.
\end{abstract}

\maketitle
\tableofcontents

\section{Introduction and main results}

\subsection{Introduction}\label{ssec:intro}

Let $f: Y \to X$ be a morphism of complex manifolds and assume that $Y$ is
compact. We aim to deform $f$, keeping $X$ and $Y$ fixed. More precisely,
given an infinitesimal deformation of $f$, say $\sigma \in H^0\bigl( Y,
f^*(T_X) \bigr)$, we ask if $\sigma$ is effective, i.e., if $\sigma$ comes
from a deformation of $f$.

It is a classical result that any infinitesimal deformation is effective if
the associated obstruction space vanishes. We refer to \cite{MR0322209}, or to
\cite[Chap.~1]{K96} for a thorough discussion of the algebraic case.

\begin{thm}\label{thm:class}
  If $H^1\bigl(Y, f^*(T_X)\bigr) = \{0\}$, then any infinitesimal deformation
  of $f$ is effective.
\end{thm}

Theorem~\ref{thm:class} is not sharp though. There are many examples of
infinitesimal deformations that are effective even though $h^1\bigl(Y,
f^*(T_X)\bigr)$ is large. In these cases, it is often possible to find a
geometric reason that explains the behavior.  Here, we consider the geometric
context where there is a subsheaf $\sF \subseteq T_X$, and where $\sigma \in
H^0 \bigl( Y, f^*(T_X)\bigr)$ is an infinitesimal deformation along $\sF$,
i.e., where $\sigma$ is in the image of the natural map
$$
H^0 \bigl( Y, f^*(\sF)\bigr) \to H^0 \bigl( Y, f^*(T_X)\bigr)
$$
If $\sF$ is closed under the Lie bracket, we show that an analogue of
Theorem~\ref{thm:class} holds for deformations along $\sF$.

The proof of our main result, Theorem~\ref{thm:mainresult}, is completely
elementary and does not use any of the sophisticated methods of deformation
theory.  The methods also illustrate the proof of Theorem~\ref{thm:class}.

\subsection{Main result}\label{ssec:mainthm}

In order to formulate the main results precisely in
Theorem~\ref{thm:mainresult} below, recall a few standard definitions and
notation used in the discussion of deformations.

\begin{defn}\label{def:defn}
  A \emph{deformation of $f$} is a holomorphic mapping $F : \Delta \times Y
  \to X$ whose restriction to $\{0\}\times Y \cong Y$ equals $f$. Here $\Delta
  \subset \mathbb C$ is a disk centered about $0$.
\end{defn}

\begin{notation}\label{not:infdef}
  If $F$ is a deformation and $t \in \Delta$ any number, we often write $F_t :
  Y \to X$ for the obvious restriction of $F$ to $\{t\}\times Y \cong
  Y$. Given a point $y \in Y$, we can consider the curve
  $$
  F_y : \Delta \to X, \quad t \mapsto F(t,y).
  $$
  Given $t \in \Delta$ and taking derivatives in $t$ for all $y$, this gives a
  section
  $$
  \sigma_{F, t} \in H^0\bigl( Y,\, (F_t)^*(T_X)\bigr),
  $$
  called ``velocity vector field at time $t$''. For $t = 0$, we obtain a
  section $\sigma_{F,0} \in H^0\bigl(Y, f^*(T_X)\bigr)$. Elements of
  $H^0\bigl(Y, f^*(T_X)\bigr)$ are thus called ``initial velocity vector
  fields'' or ``first order infinitesimal deformations of $f$''.
\end{notation}

\begin{defn}
  A first order infinitesimal deformation $\sigma \in H^0\bigl(Y,
  f^*(T_X)\bigr)$ is ``effective'' if there exists a deformation $F$ with
  $\sigma = \sigma_{F,0}$.
\end{defn}

With this notation, the main result of the present paper is formulated as
follows.

\begin{thm}[Deformation along an involutive subsheaf]\label{thm:mainresult}
  Let $f: Y \to X$ be a morphism of complex manifolds and assume that $Y$ is
  compact. Let $\sF \subseteq T_X$ be a subsheaf of $\sO_X$-modules which is
  closed under the Lie bracket, let $\sF_Y \subseteq f^*(T_X)$ be the image of
  $f^*(\sF)$ under the pull-back of the inclusion map, and let
  $$
  \sigma \in H^0 \bigl( Y,\, \sF_Y \bigr) \subseteq H^0 \bigl( Y,\, f^*(T_X)
  \bigr)
  $$
  be a first order infinitesimal deformation of the morphism $f$ that comes
  from $\sF$.

  If $H^1\bigl( Y,\, \sF_Y \bigr)= \{0\}$, then there exists a deformation $F$
  of $f$ such that $\sigma = \sigma_{F,0}$, and such that for all times $t \in
  \Delta$ the section $\sigma_{F,t}$ is in the image of
  \begin{equation}\label{eq:dalngF}
    H^0 \bigl( Y, (F_t)^*(\sF)\bigr) \to H^0 \bigl( Y, (F_t)^*(T_X)\bigr)    
  \end{equation}
\end{thm}

\begin{notation}\label{not:defalongf}
  If $F$ is any deformation of $f$ such that \eqref{eq:dalngF} holds for all
  $t$, we say that $F$ is a deformation ``along the sheaf $\sF$''.
\end{notation}

\begin{rem}
  The subsheaf $\sF \subseteq T_X$ need not be a foliation because $\sF$ need
  not be saturated in $T_X$. We recall a few special cases of
  Theorem~\ref{thm:mainresult} that we have found in the literature.
  \begin{enumerate}
  \item Foliations: The case where $\sF$ is an algebraic foliation is studied
    in Miyaoka's theory of deformation along an algebraically defined
    foliation, \cite{Miy85, MP97}.
  
  \item Logarithmic tangent sheaves: The case where $X$ contains a reduced
    divisor $D$ and $\sF = T_X(-\log D)$ appears in Keel and McKernan's work
    on the Miyanishi conjecture, \cite[Sect.~5]{KMcK}.

  \item Deformation with fixed points: A variant of the case where $\sF = T_X
    \otimes \sJ_p$ is the tangent bundle twisted with the ideal sheaf of a
    point $p$ is used in Mori's Bend-and-Break technique.
  \end{enumerate}
\end{rem}

\subsection{Outline of the paper}

In Section~\ref{sec:def}, we recall the definition of jet bundles on a complex
manifold $X$ and recall their main properties. The language of jets makes it
easy to discuss $n$th order deformations of a given morphism, and gives an
elementary way to construct classes in $H^1\bigl( Y,\, f^*(T_X)\bigr)$ that
are obstructions to extending $n$th order deformations to $n+1$st order. We
illustrate these concepts by reproducing Horikawa's proof of
Theorem~\ref{thm:class} in the language of jets, referring to Artin's paper
\cite{ARTIN68} for the necessary convergence results.

In Section~\ref{sec:strategy}, we outline the proof of
Theorem~\ref{thm:mainresult}, explain the main strategy and motivate two sets
of problems which are discussed in Sections~\ref{sec:ODEJets} and
\ref{sec:genlFrob} before completing the proof of Theorem~\ref{thm:mainresult}
in Section~\ref{sec:proof}.

Section~\ref{sec:ODEJets} concerns the relation between vector fields and
higher order jets of the integral curves they define. Given two vector fields
$D_1$ and $D_2$ on $X$ with integral curves $\gamma_1$ and $\gamma_2$, we
are interested in expressing the difference of higher order terms in the power
series expansions of the $\gamma_i$ in terms of iterated Lie brackets
involving $D_1$ and $D_2$.

In Section~\ref{sec:genlFrob}, we discuss an elementary generalization of the
classical Frobenius Theorem of Differential Geometry, where the Lie-closed
subsheaf $\sF \subseteq T_X$ is not necessarily a foliation. This will allow
us to construct local analytic subspaces of the Douady space $\Hom(Y,X)$ which
locally parametrize deformations along the sheaf $\sF$.

\subsection{Acknowledgments}

A first version of Theorem~\ref{thm:mainresult} and the elementary proof of
Theorem~\ref{thm:class} are contained in the diploma thesis of Daniel Lohmann
and Stavros Kousidis, respectively. Both thesis projects were supervised by
Stefan Kebekus.

The authors would like to thank Guido Kings for discussions, and the anonymous
\mbox{referee} for a careful review, for suggestions that helped to remove a
projectivity assumption, and for bringing Horikawa's paper to our attention.

\section{Jet bundles and deformations of morphisms}
\label{sec:def}

In Sections~\ref{ssec:tgt}--\ref{ssec:affine} we recall the definition and
briefly discuss the main properties of jet bundles of a complex manifold,
which are higher order generalizations of the tangent bundle. Jet bundles are
then used in Section~\ref{ssec:idef} to describe higher order infinitesimal
deformations of morphisms. To illustrate the use of jets in deformation
theory, we end this chapter with a short and very transparent proof of the
classical Theorem~\ref{thm:class}.

\begin{rem}
  There are two notions of ``jet bundle'' found in the literature. In this
  paper, an ``$n$-jet'' is an $n$th order curve germ. This notion was,
  originally introduced in slightly higher generality in real geometry by
  Ehresmann, cf.~\cite[Chapt.~6.29C]{MR947141}.

  Other authors use the word ``$n$-jet'' to denote an $n$th order germ of a
  section in a given vector bundle. This notion is found, e.g., in the work of
  Kumpera-Spencer on Lie equations, \cite[Chap.~1]{KS72}.
\end{rem}

\begin{notation}
  If $X$ and $Y$ are any two complex spaces where $Y$ is compact, we denote
  the Douady space of morphisms from $Y$ to $X$ by $\Hom \bigl(Y,\, X
  \bigr)$. Like the Hom-scheme of algebraic geometry, the Douady space of
  morphisms represents a functor and is therefore uniquely determined by its
  universal properties. We refer to \cite[Sect.~2]{MR1326625} for a brief
  overview and for further references.

  The reader willing to restrict himself to algebraic morphisms of projective
  varieties is free to use the Hom-scheme instead of the Douady space
  throughout this paper.
\end{notation}

\subsection{Tangent bundles}
\label{ssec:tgt}

Let $X$ be a complex manifold. Before discussing jet bundles of arbitrary
order, we recall two equivalent standard constructions of the tangent bundles
for the reader's convenience.

\begin{construction}\label{con:Tx1}
  As a manifold, the tangent bundle $T_X$ is the set of equivalence classes of
  germs of arcs $\Delta \to X$, under the equivalence relation that $\tau \sim
  \sigma$ if they agree to first order. Coordinate charts on $X$ induce
  coordinate charts on $T_X$ in the obvious canonical manner, and the map
  $\tau \mapsto \tau(0)$ induces a canonical morphism $\pi : T_X \to X$.
\end{construction}

\begin{construction}\label{con:Tx2}
  As a complex space or scheme, the tangent bundle is defined as $T_X := \Hom
  \bigl(\Spec \mathbb C[\varepsilon]/(\varepsilon^2),\, X \bigr)$, where
  $\Spec \mathbb C[\varepsilon]/(\varepsilon^2)$ denotes the double point on
  the affine line. The obvious map $\mathbb C[\varepsilon]/(\varepsilon^2) \to
  \mathbb C$ induces a canonical morphism $\pi : T_X \to X = \Hom(\Spec
  \mathbb C,\, X)$.
\end{construction}

Using either construction, an elementary computation immediately gives the
following.

\begin{fact}\label{fact:TX}
  The tangent bundle $T_X$ of a complex manifold $X$ has the structure of a
  vector bundle over $X$.

  Local coordinates on $U \subseteq X$ induce vector bundle coordinates on
  $\pi^{-1}(U) \subseteq T_X$.  More precisely, if $U \subseteq X$ is a
  coordinate neighborhood, and $\gamma$ is a germ of an arc $\gamma: \Delta
  \to U$, described in $U$-coordinates as
  $$
  \gamma(t) = \vec x_0 + \vec x_1 \cdot t + \text{(higher order terms)},
  $$
  then the associated point of $T_X$ has $\pi^{-1}(U)$-coordinates
  $(\vec x_0, \vec x_1) \in U \times \mathbb C^{\dim X}$. \qed
\end{fact}

\subsection{Jet bundles}
\label{ssec:jetbundles}

In complete analogy with Constructions~\ref{con:Tx1}--\ref{con:Tx2}, the jet
bundle of a complex manifold $X$ can be defined in one of the following
equivalent ways.

\begin{construction}
  As a manifold, the $n$th jet bundle $\Jet^n(X)$ is the set of equivalence
  classes of germs of arcs $\Delta \to X$, under the equivalence relation that
  $\tau \sim \sigma$ if they agree to $n$th order. Coordinate charts on $X$
  induce coordinate charts on $\Jet^n(X)$ in the obvious canonical manner, and
  for any $m \leq n$ the restriction of arcs to $m$th order induces a
  canonical morphism $\pi_{n,m} : \Jet^n(X) \to \Jet^m(X)$.
\end{construction}

\begin{construction}
  As a complex space or scheme the $n$th jet bundle is defined as $\Jet^n(X)
  := \Hom \bigl(\Spec \mathbb C[\varepsilon]/(\varepsilon^{n+1}),\, X
  \bigr)$. For $m \leq n$, the truncation map $\mathbb
  C[\varepsilon]/(\varepsilon^{n+1}) \to \mathbb
  C[\varepsilon]/(\varepsilon^{m+1})$ induces a canonical morphism $\pi_{n,m}
  : \Jet^n(X) \to \Jet^m(X)$.
 \end{construction}

It is clear from the construction that $\Jet^0(X) \cong X$ and $\Jet^1(X)
\cong T_X$. In complete analogy with Fact~\ref{fact:TX}, an elementary
computation in local coordinates shows the following.

\begin{fact}\label{fact:jbmp1}
  Let $X$ be a complex manifold and let $m \leq n$ be any two numbers. Then
  the following holds.
  \begin{enumerate}
  \item The morphisms $\pi_{n,m} : \Jet^n(X) \to \Jet^m(X)$ are fiber bundles,
    locally trivial in Zariski topology with fibers isomorphic to $\mathbb
    A^{(n-m)\cdot\dim X}$. In general, the transition maps are neither linear
    nor affine, and $\pi_{n,m}$ is generally neither a vector bundle nor
    an affine bundle.
    
  \item Local coordinates on $U \subseteq X$ induce vector bundle coordinates
    on $\pi_{n,0}^{-1}(U) \subseteq \Jet^n(X)$, for all $n$.  More precisely,
    if $U \subseteq X$ is a coordinate neighborhood, and $\gamma$ is a germ of
    an arc $\gamma: \Delta \to U$, described in $U$-coordinates as
    $$
    \gamma(t) = \vec v_0 + \vec v_1 \cdot t + \cdots + \vec v_n \cdot
    t^n + \text{(higher order terms)},
    $$
    then the associated point of $\Jet^n(X)$ has
    $\pi_{n,0}^{-1}(U)$-coordinates $(\vec x_0, \vec x_1, \vec x_2,
    \ldots, \vec x_n) \in U \times \mathbb C^{n\cdot \dim X}$, with
    $\vec x_i = i!\cdot \vec v_i$. In particular, the coordinate $\vec
    x_i$ is computed in local coordinates as the $i$th derivative,
    $\vec x_i = \gamma^{(i)}(0)$.
    
  \item If $m=n-1$, the fiber bundle $\pi_{m+1,m} : \Jet^{m+1}(X) \to
    \Jet^{m}(X)$ has affine transition maps and is therefore an affine
    bundle. \qed
  \end{enumerate}
\end{fact}

\subsection{Affine bundles associated with jets}
\label{ssec:affine}

We need to discuss the affine bundle structure of $\Jet^n(X) \to
\Jet^{n-1}(X)$ in more detail. For that, we briefly recall the relevant
properties of affine spaces.

\medskip

By definition, any affine space $A$ comes with a canonical vector space $V$,
the ``space of translations'', whose additive group $V$ acts on $A$. The
action map, often called ``translation map'' is usually denoted as follows,
$$
+ : V \times A \to A, \quad (\vec v, a) \mapsto \vec v + a.
$$
Given any $a \in A$, the natural map $V \to A$, $\vec v \mapsto \vec v+a$ is
an isomorphism of complex manifolds. Consequently, given any two elements $a,
b \in A$, there is a uniquely defined ``difference vector'' $\vec v \in V$,
often denoted as $\vec v = a-b$, such that $\vec v + b = a$.

In complete analogy, any affine bundle $A \to B$ naturally comes with a vector
bundle $\pi: V \to B$, the ``bundle of translations''. The translation maps on
fibers glue to give a translation map
$$
+ : V \times_B A \to A.
$$
Given any section $\sigma: B \to A$, the natural map $V \to A$, $\vec v
\mapsto \vec v + \sigma(\pi(\vec v))$ is a fiber bundle isomorphism. Consequently, given
any two sections $\sigma_1, \sigma_2: B \to A$, there is a uniquely defined
``difference section'', $\tau : B \to V$, often denoted as $\tau =
\sigma_1-\sigma_2$, such that $\tau(b) + \sigma_2(b) = \sigma_1(b)$ for all $b
\in B$.

\medskip

For the affine bundle $\Jet^{n+1}(X) \to \Jet^n(X)$, the elementary
computation used to prove Fact~\ref{fact:jbmp2} immediately identifies
the translation bundle.

\begin{fact}\label{fact:jbmp2}
  Let $X$ be a complex manifold and let $n \geq 0$ be any number. Then the
  vector bundle $V_n$ of translations associated with the affine bundle
  $\Jet^{n+1}(X) \to \Jet^n(X)$ is precisely the pull-back of the vector
  bundle $T_X$ to $\Jet^n(X)$. In other words, $V_n = \pi_{n,0}^*(T_X)$. \qed
\end{fact}

In the setup of Fact~\ref{fact:jbmp2}, if $\sigma_1, \sigma_2 : X \to
\Jet^{n+1}(X)$ are two sections that agree to $n$th order,
$\pi_{n+1,n}\circ \sigma_1 = \pi_{n+1,n}\circ \sigma_2$, then the difference
is given by a section $\sigma_1 - \sigma_2 \in H^0\bigl( X,\, T_X \bigr)$. We
will later need the following elementary generalization of this fact.

\begin{rem}\label{rem:pbdiff}
  If $f: Y \to X$ is a morphism of complex manifolds and if $\sigma_1,
  \sigma_2: Y \to f^* \Jet^{n+1}(X) = \Jet^{n+1}(X) \times_X Y$ are two
  sections in the pull-back bundles that agree to $n$th order,
  $f^*(\pi_{n+1,n})\circ \sigma_1 = f^*(\pi_{n+1,n}) \circ \sigma_2$, then the
  difference is given by a section $\sigma_1 - \sigma_2 \in H^0\bigl( Y,\,
  f^*(T_X) \bigr)$.
\end{rem}

We end this section with a remark that shows how to compute the
difference of jets in local coordinates. The (easy) proof is again
left to the reader.

\begin{rem}\label{rem:diffincoo}
  If $U \subseteq X$ is a coordinate neighborhood, and if $\gamma_1$,
  $\gamma_2 \in \Jet^{n+1}(X)$ are two jets with $\pi_{n+1,n}(\gamma_1) =
  \pi_{n+1,n}(\gamma_2)$, represented in the induced coordinates on
  $\pi_{n+1,0}^{-1}(U) \subseteq \Jet^{n+1}(X)$ as
  $$
  \gamma_1 = ( \vec x_0, \vec x_1, \ldots, \vec x_n, \vec x_{n+1,1})
  \text{\, and \,}
  \gamma_2 = ( \vec x_0, \vec x_1, \ldots, \vec x_n, \vec x_{n+1,2}),
  $$
  then the difference $\gamma_1 - \gamma_2$ is given by the tangent
  vector written in the induced coordinates on $T_X$ as
  $$
  \gamma_1 - \gamma_2 = (\vec x_0, \vec x_{n+1,1} - \vec x_{n+1,2})
  \in T_X|_{\gamma_1(0)}.
  $$
  If the base point $\gamma_1(0)$ is clear, we will often write
  $\gamma_1 - \gamma_2 = \vec x_{n+1,1} - \vec x_{n+1,2} \in
  T_X|_{\gamma_1(0)}$.
\end{rem}

\subsection{Higher order infinitesimal deformations in jet language}
\label{ssec:idef}

The following notion is the higher-order analogue of the infinitesimal
deformation discussed in the introduction.

\begin{defn}\label{def:nthod}
  Let $f: Y \to X$ be a morphism of complex manifolds. An $n$th order
  infinitesimal deformation of $f$ is a morphism
  $$
  f_n : \Spec \mathbb C[\varepsilon]/(\varepsilon^{n+1}) \times Y \to X.
  $$
  whose restriction to $Y \cong \Spec \mathbb C\times Y$ agrees with $f$.
\end{defn}

It is clear from the universal property of the Douady space of morphisms that
an $n$th order infinitesimal deformation of $f$ is the same as a morphism
$\Spec \C[\varepsilon]/(\varepsilon^{n+1}) \to \Hom(Y,X)$ that maps the closed
point to the point of $\Hom(Y,X)$ that represents $f$. For our purposes,
however, the following description is the more useful. It also shows that for
$n=1$, Definition~\ref{def:nthod} and Notation~\ref{not:infdef} agree.

\begin{prop}\label{prop:infdefn}
  To give an $n$th order infinitesimal deformation of $f$, it is equivalent to
  give a section $Y \to f^* \Jet^n(X)$, where $f^* \Jet^n(X) := \Jet^n(X)
  \times_X Y$.
\end{prop}
\begin{proof}
  It is clear from the universal property of $\Hom(Y,X)$ that to give an $n$th
  order infinitesimal deformation of $f$, it is equivalent to give a morphism
  $$
  \phi_n : Y \to \Hom\bigl(\Spec \mathbb C[\varepsilon]/(\varepsilon^{n+1}),\, X
  \bigr) = \Jet^n(X).
  $$
  with $\pi_{n,0} \circ \phi_n = f$. By the universal property of the fiber
  product, this is the same as to give a section.
\end{proof}

\subsection{Applications to deformations and to Theorem~\ref*{thm:class}}
\label{ssec:firstappl}

As an application of the methods and the language outlined in the previous
sections, we reproduce in part Horikawa's proof of Theorem~\ref{thm:class},
referring to Artin's paper \cite{ARTIN68} for the necessary convergence
results. More detailed computations are found in \cite{MR0322209}.

The proof follows the common approach to first construct a formal deformation
of $f$, which is then turned into a holomorphic solution.  The existence of a
formal solution is guaranteed by the following lemma which asserts that any
$n$th order infinitesimal deformation can be lifted to $n+1$st order.

\begin{lem}\label{lem:liftnnpo}
  In the setup of Theorem~\ref{thm:class}, let $\sigma_n : Y \to f^*\Jet^n(X)$
  be any section. Then there exists a lifting to $n+1$st order, i.e., a
  section $\sigma_{n+1} : Y \to f^*\Jet^{n+1}(X)$ making the following diagram
  commutative
  \begin{equation}\label{cd lift}
  \xymatrix{
    & & f^*\Jet^{n+1}(X) \ar[d]^{f^*(\pi_{n+1,n})} \\
    Y \ar@/^0.3cm/[urr]^{\exists \sigma_{n+1}} \ar[rr]^{\sigma_n} & & f^* \Jet^n(X).  }
  \end{equation}
\end{lem}
\begin{proof}
  Since both $f^*\Jet^n(X)$ and $f^*\Jet^{n+1}(X)$ are locally trivial on $Y$, it
  is clear that liftings to $n+1$st order always exist locally. More
  precisely, there exists a covering of $Y$ with open sets
  $(U_{\alpha})_{\alpha \in A}$ and there are sections $\sigma^\alpha_{n+1} :
  U_\alpha \to f^*\Jet^{n+1}(X)$ such that $f^*(\pi_{n+1,n}) \circ
  \sigma^\alpha_{n+1} = \sigma_n|_{U_\alpha}$. We have seen in
  Remark~\ref{rem:pbdiff} that for any $\alpha, \beta \in A$, the difference
  defines a section
  $$
  \nu_{\alpha \beta} = \sigma^\alpha_{n+1}|_{U_\alpha \cap U_\beta} -
  \sigma^\beta_{n+1}|_{U_\alpha \cap U_\beta} \in H^0\bigl( U_\alpha \cap U_\beta,
  f^*(T_X) \bigr).
  $$
  The $\nu_{\alpha \beta}$ obviously satisfy the {\v C}ech cocycle condition and we
  obtain a cohomology class $(\nu_{\alpha \beta}) \in H^1\bigl(Y,\, f^*(T_X)
  \bigr)$ which is zero by assumption.

  Consequently, there are sections $\lambda_\alpha \in H^0\bigl( U_\alpha,\,
  T_X\bigr)$ with $\lambda_\alpha - \lambda_\beta = \nu_{\alpha \beta}$. If we
  set
  $$
  \sigma'^\alpha_{n+1} :=  (-\lambda_\alpha) + \sigma^\alpha_{n+1},
  $$
  then $\sigma'^\alpha_{n+1}$ and $\sigma'^\beta_{n+1}$ agree on
  $U_\alpha \cap U_\beta$ for any $\alpha, \beta \in A$ and therefore
  define a global section $\sigma_{n+1} : Y \to f^*\Jet^{n+1}(X)$ that
  lifts $\sigma_n$.
\end{proof}

\begin{proof}[Proof of Theorem~\ref{thm:class}]
  Let $\sigma \in H^0\bigl( Y,\, f^*(T_X)\bigr)$ be any first order
  infinitesimal deformation. Choose a neighborhood $U$ of the point $[f] \in
  \Hom\bigl( Y,\, X\bigr)$, and view $U$ as a subset of $\mathbb A^n$, given
  by equations $U = \{ f_1 = \cdots = f_m = 0 \}$. With this notation, our aim
  is to find a holomorphic map $\tilde \sigma : \Delta \to \mathbb A^n$ which
  agrees with $\sigma$ to first order and satisfies $f_i \circ \tilde \sigma =
  0$ for all $i$. By Michael Artin's result on solutions of analytic
  equations, \cite[Thm.~1.2]{ARTIN68}, a holomorphic solution will exist if
  there is a formal solution to the problem.

  Using Lemma~\ref{lem:liftnnpo} inductively, we can find a sequence $\sigma = \sigma_1,
  \sigma_2, \ldots$ of liftings to arbitrary order, with $\pi_{n+1,n} \circ \sigma_{n+1}
  = \sigma_n$. If we view the $\sigma_n$ as morphisms
  $$
  \sigma_n: \Spec \mathbb C[\varepsilon]/(\varepsilon^{n+1}) \to \Hom\bigl( Y,\,
  X\bigr),
  $$
  this defines a formal map
  $$
  \hat \sigma: \Spec \mathbb C[[\varepsilon]] \to \Hom\bigl( Y,\, X\bigr),
  $$
  which satisfies $f_i \circ \hat \sigma = 0$ for all $i$, and whose first order
  part agrees with $\sigma$. Artin's result therefore applies.
\end{proof}

\section{Strategy for the proof of Theorem~\ref*{thm:mainresult}}
\label{sec:strategy}

\subsection{Introduction}

Before giving a complete proof of Theorem~\ref{thm:mainresult} in
Section~\ref{sec:proof} below, we first outline the main strategy of the proof
and recall a few elementary facts. We hope that the explanations given below
will help to motivate the preparatory Sections~\ref{sec:ODEJets} and
\ref{sec:genlFrob} where we gather several technical results used in the
proof.

We will constantly use a number of elementary facts concerning vector fields
on mani\-folds, their associated ordinary differential equations, flow maps and
local actions of 1-parameter groups. Since all relevant results hold without
change in the holomorphic as well as in the $C^\infty$ category, we have
chosen to use \cite{Warner83} as our main reference, for the reader's
convenience. A thorough introduction to vector fields and their flows on
possibly singular complex spaces is found in \cite{Kaup65}.

\subsection{Outline of the proof}
\label{ssec:outline}

To start the outline, consider the setup of Theorem~\ref{thm:mainresult} in
the simple case where $f: Y \to X$ is a closed immersion and where both $X$
and $Y$ are compact. Viewing $Y$ as a subspace of $X$, let
$$
\sigma \in \Image \Bigl(H^0\bigl(Y,\, \sF|_Y \bigr) \to H^0\bigl(Y,\, T_X|_Y
\bigr) \Bigr)
$$
be a first order infinitesimal deformation of $f$ along $\sF$.

If $\sigma$ is the restriction of a global vector field $D \in H^0\bigl( X,\,
\sF \bigr)$, we can integrate the vector field $D$ globally on $X$, obtaining
a holomorphic action of a 1-parameter group, say
$$
\phi : \Delta \times X \to X,
$$
such that for each point $x \in X$, the arc $\gamma_x : \Delta \to X$, $t
\mapsto \phi(t, x)$ is a solution to the initial value problem associated with
the ordinary differential equation described by $D$.  In down-to-earth terms,
the germ of $\gamma_x$ is the unique solution to the problem of finding a germ
of an arc $\gamma : \Delta \to X$ that satisfies the two following
requirements,
\begin{align}
  \gamma(0) & = x && \text{and} \label{eq:ivp0}\\
  \gamma'(t)& = D\bigl(\gamma(t) \bigr) && \text{for all } t \in \Delta. \label{eq:ivp}
\end{align}

\begin{notation}\label{not:intcurve}
  We call $\gamma_x$ the \emph{integral curve of $D$ through $x$}.
\end{notation}

Viewing $\phi$ as a deformation of $f$, this gives a proof of
Theorem~\ref{thm:mainresult} in case $\sigma$ comes from a global vector
field. For this, observe that requirement~\eqref{eq:dalngF} of
Theorem~\ref{thm:mainresult} immediately follows from \eqref{eq:ivp} above.

If $\sigma$ is the restriction of a vector field $D \in H^0\bigl( U,\, \sF
\bigr)$ that is defined only on an open neighborhood $U$ of $Y$, but perhaps
not on all of $X$, essentially the same strategy applies.  In this setup,
there exists a \emph{local} action, cf.~\cite[Thm.~1.48]{Warner83}. More
precisely, there exists an open, relatively compact neighborhood $V$ of $Y$
with $Y \subseteq V \Subset U$, there exists a disk $\Delta$ and a map
$$
\phi : \Delta \times V \to U,
$$
such that the arcs $t \mapsto \phi(t, x)$ are again solutions to the initial
value problems~\eqref{eq:ivp}. As before $\phi$ gives a deformation of $f$
that solves the problem.

In general, however, $\sigma$ is not the restriction of a vector field that
lives on a neighborhood of $Y$, and extensions of $\sigma$ to open subsets of
$X$ exist only locally, cf.~\cite[Rem.~1.52]{Warner83}. More precisely, there
exist finitely many open sets $U_i$ that are open in $X$, cover $Y$ and admit
vector fields $D_i \in H^0 \bigl( U_i,\, T_X \bigl)$ whose restrictions $D_i
|_{Y \cap U_i}$ equal $\sigma|_{Y \cap U_i}$. As before, we find relatively
compact open subsets $V_i \Subset U_i$ that still cover $Y$, and local action
morphisms
$$
\phi_i : \Delta_i \times V_i \to U_i,
$$
again with the property that if $x$ is a point in $V_i$, we obtain an arc
$\gamma_{x,i} : \Delta \to X$ that solves the initial value problem for $D_i$,
as in~\eqref{eq:ivp0} and \eqref{eq:ivp} above.  However, if $i \not = j$ are
any two indices, the local action morphisms will generally not agree on the
overlap $V_i \cap V_j$, and if $x$ is in $V_i \cap V_j \cap Y$, the arcs
$\gamma_{x,i}$ and $\gamma_{x,j}$ will likewise not agree.

There are a few things we can say about $\gamma_{x,i}$ and $\gamma_{x,j}$,
though. Since
$$
D_i|_{V_i \cap V_j \cap Y} = D_j|_{V_i \cap V_j \cap Y} = \sigma|_{V_i \cap
  V_j \cap Y}
$$
and since $\gamma_{x,i}$ and $\gamma_{x,j}$ satisfy~\eqref{eq:ivp}, it is
clear that for any point $x \in V_i \cap V_j \cap Y$, the arcs $\gamma_{x,i}$
and $\gamma_{x,j}$ agree to first order, though perhaps not to second order.
In other words, the $\phi_i$ induce sections\footnote{Since the $\gamma_x$ are
  holomorphic for each $x$, the $\phi_i$ give sections in
  $\Jet^n(X)|_{V_\bullet \cap Y}$, for any number $n$. For the purposes of
  this outline, we concentrate on the case $n = 2$.}
$$
\tau_{D_i}^2 : V_i \cap Y \to \Jet^2(X)|_{V_i \cap Y} \text{\, and \,}
\tau_{D_j}^2 : V_j \cap Y \to \Jet^2(X)|_{V_j \cap Y}
$$
whose first-order parts $\pi_{2,1}(\tau_\bullet) : V_\bullet \cap Y \to
\Jet^1(X)|_{V_\bullet \cap Y}$ agree on the overlap $V_i \cap V_j \cap Y$. We
have seen in Section~\ref{ssec:affine} that the difference $\tau_{D_j}^2 -
\tau_{D_i}^2$ can be expressed as a section of $T_X|_{V_i \cap V_j \cap Y}$,
and we will see in Theorem~\ref{thm:dj} below that this difference is
expressed in terms of the Lie bracket of the vector fields $D_\bullet$, as
follows
$$
\tau^2_{D_j}|_{V_i \cap V_j \cap Y} - \tau^2_{D_i}|_{V_i \cap V_j \cap Y} =
\bigl[D_i,\, D_j\bigr]|_{V_i \cap V_j \cap Y}.
$$
This will allow us to describe the {\v C}ech cocycles associated with the
problem of lifting the infinitesimal deformation $\sigma$ from first to second
order in terms of Lie brackets. An argument similar to the proof of
Lemma~\ref{lem:liftnnpo} will then allow to adjust the vector fields $D_i$, in
a way such that the associated local group actions give a well-defined lifting
of $\sigma$ to second order, globally along $Y$. An iterated application of
this method will give liftings to arbitrary order.

\section{Jets associated with vector fields}
\label{sec:ODEJets}

If $D_1$ and $D_2$ are two vector fields on $X$ and $x \in X$ is a
point, the integral curves $\gamma_i$ of $D_i$ through $x$ do
generally not agree. If $\gamma_1$ and $\gamma_2$ agree to $n$th
order, we have seen that the difference between the $n+1$st order
parts of the $\gamma_i$ can be expressed as an element $\vec v \in
T_X|_x$. In this section, we aim to express $\vec v$ purely in terms
of the vector fields $D_i$ and their Lie brackets. Before formulating
the result in Theorem~\ref{thm:dj} below, we need to introduce
notation.

\begin{defn}[Jets associated with a vector field]\label{def:jetvec}
  Let $U \subseteq X$ be an open set, and let $D \in H^0\bigl( U,\, T_X
  \bigr)$ be a vector field. Given any number $n \in \mathbb N$, let $\tau^n_D
  : U \to \Jet^n(X)$ be the section in the $n$th jet bundle induced by the
  local action of the vector field.

  In other words, if $x \in U$ is any point, and $\gamma_x : \Delta \to X$ the
  unique curve germ that satisfies~\eqref{eq:ivp0} and \eqref{eq:ivp}, then
  $\tau^n_D(x)$ is exactly the $n$th order jet associated with $\gamma_x$.
\end{defn}

\begin{defn}[Iterated Lie brackets]
  Let $U \subseteq X$ be an open set, and let $D_1, D_2 \in H^0\bigl( U,\, T_X
  \bigr)$ be two vector fields. If $n\in \bN$ is any number, $n \geq 2$, we
  recursively define a vector field, called the $n$th iterated Lie bracket of
  $D_1$ and $D_2$, as follows,
  $$
  \left[D_1,\, D_2\right]^{(2)} := \left[D_1,\, D_2\right] \text{\,\, and
    \,\,} \left[D_1,\, D_2\right]^{(n)} := \bigl[D_1,\,\left[D_1 ,\,
    D_2\right]^{(n-1)} \bigr].
  $$
\end{defn}

\begin{thm}\label{thm:dj}
  Let $U \subseteq X$ be an open set, and let $D_1, D_2 \in H^0 \bigl( U,\,
  T_X \bigr)$ be two vector fields. If $x \in U$ is any point and $n$ any
  number such that the $n$th order jets associated with $D_1$ and $D_2$ agree
  at $x$, i.e.~$\tau^n_{D_1}(x) = \tau^n_{D_2}(x)$, then the tangent vector
  that describes the difference between the $n+1$st order jets is expressed in
  terms of iterated Lie brackets as follows,
  \begin{equation}\label{eq:t33}
    \tau_{D_2}^{n+1}(x) - \tau_{D_1}^{n+1}(x) = [D_1,\,D_2]^{(n+1)}|_x.    
  \end{equation}
\end{thm}

\begin{proof}[Proof of Theorem~\ref{thm:dj} for $n=1$]
  Choose a coordinate neighborhood $U$ of $x$ and let $\gamma_i: \Delta \to X$
  be the germs of the integral curves of $D_i$ through $x$ for $i \in
  \{1,2\}$.  By Remark~\ref{rem:diffincoo} and Fact~\ref{fact:jbmp1}, the
  difference between the second order parts of the $\gamma_i$ is then
  expressed in $U$-coordinates as the difference of the second derivatives,
  \begin{equation}\label{eq:dtau}
    \vec v := \tau_{D_2}^2(x) - \tau_{D_1}^2(x) = \gamma_2''(0) -
    \gamma_1''(0) \in T_X|_x. 
  \end{equation}
  We aim to express the right hand side of~\eqref{eq:dtau} in terms of the
  vector fields $D_i$. For that, it is convenient to recall that to give a
  vector field $D$ on $U$, it is equivalent to give a derivation $\sO_X|_U \to
  \sO_X|_U$, written as $f \mapsto Df$. Likewise, to give a tangent vector at
  $x$, it is equivalent to give a derivation $\sO_{X,x} \to \mathbb C$, where
  $\sO_{X,x}$ denotes the stalk of $\sO_X$ at $x$. For a given tangent vector
  $\vec w \in T_X|_x$, the derivation is $f \mapsto f'(x)\cdot \vec w$, where
  $f'$ is the derivative of $f$ in $U$-coordinates, and the dot is
  matrix-vector multiplication.  The derivations commute with restriction, so
  that $(Df)(x) = f'(x) \cdot D|_x$ for all $f$.

  Now, if $f \in \sO_{X,x}$ is any germ of a function, taking the second
  derivative of $f\circ \gamma_i$ yields
  \begin{equation}\label{eq:dt1}
    f'(x)\cdot \vec v = f'(x)\cdot \bigl( \gamma_2''(0) - \gamma_1''(0)
    \bigr) = (f\circ \gamma_2 - f\circ \gamma_1)''(0).
  \end{equation}
  In order to relate the right hand side of~\eqref{eq:dt1} to the vector
  fields $D_i$, recall Equation~\eqref{eq:ivp}, which asserts that for any
  function $g$, we have $(g\circ \gamma_i)' = (D_i g) \circ \gamma_i$.
  Applying this to $f$ and $D_if$, we obtain the following expression for the
  second derivatives of $f \circ \gamma_i$,
  \begin{equation}\label{eq:dt3}
    (f\circ \gamma_i)'' = \bigl( (D_if) \circ \gamma_i \bigr)' =
    \bigl(D_i(D_if)\bigr) \circ \gamma_i = (D_i^2f) \circ \gamma_i.
  \end{equation}
  Substituting~\eqref{eq:dt3} into \eqref{eq:dt1} we find that the equality
  $f'(x)\cdot \vec v = \bigl((D_2^2-D_1^2)f \bigr)(x)$ holds for all $f \in
  \sO_{X,x}$, and therefore expresses $\vec v$ in terms of the vector fields
  $D_i$.  To prove~\eqref{eq:t33}, it is therefore sufficient to show that
  \begin{equation}\label{eq:dt5}
    \bigl((D_2^2-D_1^2)f \bigr)(x) = \bigl([D_1, D_2]f \bigr)(x) =
    \bigl((D_1D_2-D_2D_1)f \bigr)(x)    
  \end{equation}
  holds true for all $f \in \sO_{X,x}$.  We show a stronger statement: for all
  $f$ we have
  \begin{equation}\label{eq:dt6}
    (D_1^2f)(x) = (D_2D_1f)(x) \text{\, and \,} (D_2^2f)(x) =
    (D_1D_2f)(x).
  \end{equation}
  To prove \eqref{eq:dt6}, note that the equality $D_1|_x = D_2|_x$ implies
  that $(D_1g)(x) = (D_2g)(x)$ for every $g \in \sO_{X,x}$. An application to
  $g = D_1f$ and $g = D_2f$, respectively, gives the two equalities
  in~\eqref{eq:dt6}.
\end{proof}

\begin{proof}[Sketch of proof of Theorem~\ref{thm:dj} for arbitrary $n$]
  The line of argumentation used to show Theorem~\ref{thm:dj} in case $n = 1$
  also works for arbitrary $n$. As a first step, one shows that the difference
  vector $\vec v := \tau_{D_2}^{n+1}(x) - \tau_{D_1}^{n+1}(x) =
  \gamma_2^{(n+1)}(0) - \gamma_1^{(n+1)}(0) \in T_X|_x$ is determined by that
  fact that it satisfies the equation
  \begin{equation}\label{eq:dt7}
    f'(x)\cdot \vec v = \bigl((D_2^{n+1}-D_1^{n+1})f \bigr)(x)    
  \end{equation}
  for all $f \in \sO_{X,x}$. Once this is established, it remains to show that
  \begin{equation}\label{eq:dt8}
    \bigl((D_2^{n+1}-D_1^{n+1})f \bigr)(x) = \bigl([D_1, D_2]^{(n+1)} f \bigr)(x),
  \end{equation}
  again for all $f \in \sO_{X,x}$. Equations~\eqref{eq:dt7} and \eqref{eq:dt8}
  can be shown by induction on $n$, using elementary but tedious computations
  in local coordinates. We refer to \cite[Satz~1.4]{Lohmann08} for details.
\end{proof}

\section{Frobenius theorems and deformations along subsheaves}
\label{sec:genlFrob}

In Theorem~\ref{thm:mainresult}, we aim to deform the morphism $f$ along the
sheaf $\sF$. For that, we aim to define an analytic subspace
$\Hom_{\sF}\bigl(Y,\, X \bigr) \subseteq \Hom\bigl(Y,\, X \bigr)$ which
parametrizes such deformations. If $\sF$ is a regular foliation, the space
$\Hom_{\sF}\bigl(Y,\, X \bigr)$ can be defined as a relative analytic Douady
space of morphisms, using the classical Frobenius Theorem which asserts that
$\sF$ is the foliation associated with a morphism, at least locally.

\begin{thm}[\protect{Frobenius Theorem, cf.~\cite[Thm.~1.60]{Warner83}}]\label{thm:frob}
  Let $Z$ be a complex manifold and $\sG \subset T_Z$ a regular foliation,
  i.e., a vector subbundle of $T_Z$ which is closed under the Lie bracket.  If
  $z \in Z$ is any point, then there exists an analytic neighborhood $U = U(z)
  \subset Z$ which has a product structure, $U = A \times B$, such that $\sG =
  \pi_A^*(T_A)$, where $\pi_A : A\times B \to A$ is the projection to the
  first factor. \qed
\end{thm}

After introducing some notation and after proving the auxiliary
Proposition~\ref{prop:flowstabxr}, we give a generalization of the Frobenius
Theorem that works for arbitrary Lie-closed sheaves. While this result,
formulated in Corollary~\ref{cor:newfrob}, is probably known to experts, we
include a full proof, for lack of an adequate reference. We will use this
version of the Frobenius Theorem to define the space $\Hom_{\sF}\bigl(Y,\, X
\bigr)$ in Corollary~\ref{cor:analyticset} and to prove some of its universal
properties.

Throughout the present section, we maintain the notation of
Theorem~\ref{thm:mainresult} where $X$ is a complex manifold and $\sF
\subseteq T_X$ a sheaf which is closed under the Lie bracket.

\begin{notation}[Stratification of $X$]\label{not:strata}
  It follows immediately from semicontinuity of rank that for any number $r$,
  the subset
  $$
  X_r := \bigl\{ x \in X\, | \, \rank( \sF|_x \to T_X|_x) = r \bigr\}
  \subseteq X
  $$
  is a locally closed analytic subspace of $X$. We consider the natural
  sequence of closed analytic subspaces of $X_r$,
  $$
  X_r = X_r^0 \supseteq X_r^1 \supseteq \cdots \supseteq X_r^{m_r} \supseteq
  X_r^{m_r+1} = \emptyset,
  $$
  where $X_r^{i+1}$ is defined inductively as the singular locus of $X_r^{i}$.
  We obtain a decomposition of $X$ into finitely many disjoint, smooth and
  locally closed analytic subspaces,
  $$
  X = \bigcup^{\bullet}_{r,s} Z_r^s \text{\,\, with \,\,} Z_r^s := X_r^s
  \setminus X_r^{s+1}.
  $$
\end{notation}

\begin{prop}\label{prop:flowstabxr}
  Let $r$ be any number such that $X_r \not = \emptyset$, let $x \in X_r$ be
  any point and $D \in H^0 \bigl( U,\, \sF \bigr)$ a vector field, defined in
  a neighborhood $U = U(x) \subseteq X$ of $x$. If $\gamma_x : \Delta \to X$
  is the integral curve of $D$ through $x$, as defined in
  Notation~\ref{not:intcurve}, then $\gamma_x(t) \in X_r$ for all $t \in
  \Delta$.
\end{prop}
\begin{proof}
  Let $q \leq r$ be the minimal number such that the set $\Delta_q :=
  \gamma_x^{-1}(X_q)$ is not empty. By semicontinuity, $\Delta_q \subseteq
  \Delta$ is a closed analytic subset, and to prove
  Proposition~\ref{prop:flowstabxr}, it suffices to show that $\Delta_q$ is
  also open. Using the fundamental property that $\gamma_{\gamma(t_0)}(t) =
  \gamma_x(t+t_0)$ for all $t_0 \in \Delta_q$ and all sufficiently small
  numbers $t$, we can assume without loss of generality that $0 \in \Delta_q$
  and $r=q$. For the same reason, it suffices to show that $\Delta_q$ contains
  a neighborhood $\Delta'$ of $0 \in \Delta$.

  To that end, we will show that near $0 \in \Delta$, the local group action
  induced by $D$ yields an injective linear map from
  $\Image(\sF|_{\gamma_x(t)} \to T_X|_{\gamma_x(t)})$ to a $q$-dimensional
  vector space, for every sufficiently small number $t$. Shrinking $U$, if
  necessary, we can assume without loss of generality that the sheaf
  $\sF|_{U}$ is generated by vector fields $D_1,\dots, D_s \in H^0 \bigl( U,\,
  \sF \bigr)$. The vector field $D$ induces a local group action $\phi :
  \Delta' \times V \to U$, where $V \subseteq U$ and $\Delta' \subseteq
  \Delta$ are suitably small open neighborhoods of $x$ and $0$, respectively.
  
  To prove Proposition~\ref{prop:flowstabxr}, we need to show that $\Delta'
  \subseteq \Delta_q$. For this, pick any element $t \in \Delta'$ and set $y
  := \gamma_x(t)$. We consider the vector spaces
  $$
  W_x :=\left < D_1(x),\dots,D_s(x)\right > \subseteq T_X|_x \quad \mbox{and}
  \quad W_y :=\left < D_1(y),\dots,D_s(y)\right > \subseteq T_X|_y.
  $$
  Since $W_x = \Image(\sF|_x \to T_X|_x)$, the dimension of $W_x$ equals
  $r=q$, and since $q$ is chosen minimal, Proposition~\ref{prop:flowstabxr} is
  shown once we prove that $\dim W_y \leq q$.  In order to relate the spaces
  $W_x$ and $W_y$ we consider the open immersion $\phi_t : V \to U$,
  $\phi_t(v) := \phi(t, v)$, whose pull-back morphism yields an isomorphism of
  vector spaces $\phi_t^* : T_X|_y \to T_X|_x$. 

  To understand the morphism $\phi_t^*$ better, let $D_i(y)$ be any generator
  of $W_y$, and define the map
  $$
  \begin{tabular}{rccc}
    $\Gamma$ : & $\Delta'$ & $\to$ & $T_X|_x$ \\
    & $t'$ & $\mapsto$ & $(\phi_{t'}^* \circ D_i) \bigl(\phi_{t'}(x) \bigr)$
  \end{tabular}
  $$
  and notice that $\Gamma(0) = D_i(x)\in W_x$ and $\Gamma(t) = \phi_{t}^*
  \bigl( D_i(y) \bigr)$. Since $\phi_t^*$ is injective, it remains to prove
  that $\Gamma(t)$ is an element of $W_x$. Recall from \cite[Def.~2.24,
  Prop~2.25]{Warner83} that $\Gamma$ is analytic and that its derivative is
  $$
  \Gamma'(t') = (\phi_{t'}^*\circ[D,\,D_i]) \bigl(\phi_{t'}(x)\bigr).
  $$
  In particular, we have that $\Gamma'(0) = [D,\,D_i](x)$ is an element of
  $W_x$.  It follows by induction that the higher order derivatives are given
  by
  $$
  \Gamma^{(n)}(t') = (\phi_{t'}^*\circ[D ,\,D_i]^{(n)})\bigl(\phi_{t'}(x)
  \bigr).
  $$
  In particular, we have that $\Gamma^{(n)}(0) = [D,\,D_i]^{(n)}(x)$ is an
  element of $W_x$, for all numbers $n$.  Expanding $\Gamma$ in a Taylor
  series, it follows that $\Gamma(t')$ is an element of $W_x$, for all $t' \in
  \Delta'$. 

  In summary, we see that the isomorphism $\phi_t^* : T_X|_y \to T_X|_x$ maps
  each generator $D_i(y)$ of $W_y$ to $W_x$. As a consequence, we obtain $\dim
  W_y \leq \dim W_x = q$, as claimed. This ends the proof of
  Proposition~\ref{prop:flowstabxr}.
\end{proof}

\begin{cor}[Frobenius Theorem for $\sF$]\label{cor:newfrob}
  If $r$, $s$ are any two numbers such that $Z := Z_r^s$ is not empty, then
  \begin{enumerate}
  \item\ilabel{il:AA} the image of $\sF$ along $Z$ is contained in the tangent
    bundle of $Z$, i.e.,
    $$
    \sF_Z := \Image( \sF|_Z \to T_X|_Z) \subseteq T_Z,
    $$
  \item\ilabel{il:BB} the sheaf $\sF_Z \subseteq T_Z$ is a regular foliation,
    and
  \item\ilabel{il:CC} every point $z \in Z$ admits an open neighborhood $U =
    U(z) \subseteq Z$ with a product structure, $U = A \times B$ such that
    $\sF_Z \cong \pi_A^*(T_A)$, where $\pi_A : A \times B \to A$ is the
    projection to the first factor.
  \end{enumerate}
\end{cor}
\begin{proof}
  Let $U \subseteq X$ be any open subset of $X$, and let $D \in H^0 \bigl(
  U,\, \sF \bigr)$ be any vector field, with an associated local group action
  $\phi: \Delta \times V \to X$, where $\Delta$ is again a sufficiently small
  disk and $V \subseteq U$ a suitable open subset that contains $x$. By
  Proposition~\ref{prop:flowstabxr}, we know that for any point $x' \in V$ and
  any $t \in \Delta$, we have $\phi(t, x') \in X_r$ if and only if $x' \in
  X_r$. In fact, more is true: since the morphisms $\phi(t, \cdot) : V \to X$
  are open immersions, they must stabilize the singular set of $X_r$.
  Eventually, it follows that for any number $s$, we have $\phi(t, x') \in
  X_r^s$ if and only if $x' \in X_r^s$. Since
  $$
  D|_x = \bigl( \textstyle \frac{\partial}{\partial t} \phi\bigr)(0, x) =
  \bigl(\textstyle \frac{\partial}{\partial t} \gamma_x \bigr)(0),
  $$
  this implies Claim~\iref{il:AA}.

  By definition of $X_r$, it is clear that $\sF_Z$ is a vector subbundle of
  $T_Z$. The assertion that $\sF_Z$ is closed under the Lie bracket of $T_Z$
  follows from Claim~\iref{il:AA} and a standard computation,
  cf.~\cite[Prop.~1.55]{Warner83}, giving
  Claim~\iref{il:BB}. Claim~\iref{il:CC} follows when one applies the
  classical Frobenius Theorem~\ref{thm:frob} to $\sF_Z \subseteq T_Z$.
\end{proof}

Using Corollary~\ref{cor:newfrob}, we can now define the analytic space
$\Hom_{\sF}\bigl(Y,\, X \bigr)$ which parametrizes deformations along
$\sF$. The following notation is useful for the description of its universal
properties.

\begin{defn}[Infinitesimal deformations that are pointwise induced by a subsheaf]\label{def:pibss}
  Let $\sigma_n : Y\to f^*\Jet^n(X)$ be an $n$th order infinitesimal
  deformation of the morphism $f$. We say that $\sigma_n$ is \emph{pointwise
    induced by vector fields in $\sF$}, if for any point $y \in Y$ there is a
  neighborhood $U \subseteq X$ of $f(y)$ and a vector field $D \in H^0\bigl(U,
  \, \sF \bigr)$ such that $\sigma_n(y) = \tau^n_D\bigl( f(y) \bigr)$, where
  $\tau^n_D : U \to \Jet^n(X)$ is the section in the $n$th jet bundle
  described in Definition~\ref{def:jetvec}.
\end{defn}

\begin{cor}[Existence of a parameter space for deformations along a subsheaf]\label{cor:analyticset}
  There exists a locally closed analytic subspace $\Hom_{\sF}\bigl(Y,\, X
  \bigr) \subseteq \Hom\bigl(Y,\, X \bigr)$ which contains the morphism $f$
  and has the following universal properties.
  \begin{enumerate-p}
  \item\label{il:laurel} If $\sigma_n$ is an $n$th order infinitesimal
    deformation of the morphism $f$ which is pointwise induced by vector
    fields in $\sF$, then the associated morphism $\Spec \mathbb
    C[\varepsilon]/(\varepsilon^{n+1}) \to \Hom\bigl(Y,\, X \bigr)$ factors
    via $\Hom_{\sF}\bigl(Y,\, X\bigr)$.
  \item\label{il:hardy} If $\gamma: \Delta \to \Hom_{\sF}\bigl(Y,\, X \bigr)$
    is any arc with $\gamma(0) = f$, and if $F: \Delta \times Y \to X$ is the
    associated deformation, then $F$ is a deformation along $\sF$, in the
    sense of Notation~\ref{not:defalongf}.
  \end{enumerate-p}
\end{cor}
\begin{proof}
  We begin with the construction of the space $\Hom_{\sF}\bigl(Y,\, X\bigr)$.
  Choose numbers $r$, $s$ with $f(Y) \cap Z_r^s \not = \emptyset$, an
  irreducible component $Y' \subseteq f^{-1}(Z_r^s)$, a general point $y_0 \in
  Y'$ and a neighborhood $V = V \bigl(f(y_0)\bigr) \subseteq X$, with a
  decomposition $V \cap Z_r^s = A \times B$ as in
  Corollary~\ref{cor:newfrob}. Let $U= U(y_0) \Subset Y' \cap f^{-1}(V)$ be a
  relatively compact neighborhood. By relative compactness of $U$, there
  exists an analytically open neighborhood $H_{r,s, Y'}^1 \subseteq \Hom\bigl(
  Y,\, X \bigr)$ of $f \in \Hom\bigl( Y,\, X \bigr)$ such that $g(y) \in V$
  for all points $y \in U$ and all morphisms $g \in H_{r,s, Y'}^1$. The set
  \begin{equation}\label{eq:valentin}
    H_{r,s, Y'}^2 := \bigcap_{y \in U} \bigl\{ g \in H_{r,s, Y'}^1 \,|\, g(y)
    \in Z_r^s \bigr\} \subseteq H_{r,s, Y'}^1
  \end{equation}
  is then the intersection of finitely or infinitely many analytic subspaces,
  and therefore, by the analytic version of Hilbert's Basissatz
  \cite[Prop.~23.1]{KaupKaup}, an analytic subspace itself.  We remark that
  neither $H_{r,s,Y'}^2$ nor any of the spaces on the right hand side
  of~\eqref{eq:valentin} are necessarily reduced.

  Identifying $V \cap Z_r^s \cong A\times B$, with projection $\pi_B: A \times
  B \to B$, we can then consider the following analytic subspace of
  $H_{r,s,Y'}^2$,

  \begin{equation}\label{eq:ssp}
    H_{r,s, Y'} := \bigcap_{y \in U} \bigl\{ g \in H_{r,s, Y'}^2 \,|\,
    (\pi_B \circ g)(y) = (\pi_B \circ f)(y) \bigr\} \subseteq H_{r,s, Y'}^2.
  \end{equation}

  In order to define the subspace $\Hom_{\sF}\bigl(Y,\, X \bigr) \subseteq
  \Hom\bigl(Y,\, X \bigr)$, repeat this construction for each of the finitely
  many numbers $r$ and $s$, and for each of the finitely many components $Y'
  \subseteq f^{-1}(Z_r^s)$. Finally, let $\Hom_{\sF}\bigl(Y,\, X \bigr)$ be
  the connected component of the intersection which contains $f$,
  $$
  \Hom_{\sF}\bigl(Y,\, X \bigr) \subseteq \bigcap_{r,s,Y'} H_{r,s, Y'}
  \subseteq \underbrace{\bigcap_{r,s,Y'} H_{r,s, Y'}^1}_{\text{open in
      $\Hom(Y, X)$}} \subseteq \Hom\bigl(Y,\, X \bigr).
  $$
  It remains to show that the Universal Properties~(\ref{cor:analyticset}.\ref{il:laurel}) and
  (\ref{cor:analyticset}.\ref{il:hardy}) hold.

  \medskip

  For Property~(\ref{cor:analyticset}.\ref{il:laurel}), assume that an $n{\rm th}$ order deformation
  $\sigma_n$ is given as in (\ref{cor:analyticset}.\ref{il:laurel}). Given any two numbers $r$, $s$
  and any connected component $Y' \subseteq f^{-1}(Z_r^s)$, let $V \subseteq
  X$ and $U \Subset Y' \cap f^{-1}(V)$, be the sets considered above in the
  construction of $H_{r,s, Y'}$, with decomposition $V \cap Z_r^s \cong A
  \times B$. Now, if $y \in U$ is any point and $D$ the associated vector
  field near $f(y)$, with integral curve $\gamma_{f(y)}: \Delta \to X$, it is
  clear from Corollary~\ref{cor:newfrob} that $\gamma_{f(y)}(t) \in Z_r^s$,
  for all $t$. In particular, the associated morphism
  \begin{equation}\label{eq:lst}
    \sigma_n : \Spec \mathbb C[\varepsilon]/(\varepsilon^{n+1}) \to
    \Hom\bigl( Y,\, X\bigr)
  \end{equation}
  factors via $H_{r,s, Y'}^2$. In a similar vein, it follows from
  Corollary~\ref{cor:newfrob} that $\pi_B \circ \gamma_{f(y)}(t) \equiv \pi_B
  \bigl( f(y) \bigr)$ for all $t \in \Delta$. In particular, viewing
  $\sigma_n$ as a map $\sigma_n : \Spec \mathbb
  C[\varepsilon]/(\varepsilon^{n+1})\times Y \to X$, we have
  $$
  \pi_B \circ \bigl( \sigma_n|_{\Spec \mathbb C[\varepsilon] /
    (\varepsilon^{n+1})\times \{y\}} \bigr) \equiv \pi_B \bigl( f(y) \bigr),
  $$
  so that the morphism~\eqref{eq:lst} actually factors via $H_{r,s,
    Y'}$. Since this is true for all $r$, $s$ and $Y'$, the
  morphism~\eqref{eq:lst} factors via $\Hom_{\sF}\bigl(Y,\, X\bigr)$, as
  claimed. This ends the proof of Property~(\ref{cor:analyticset}.\ref{il:laurel}).

  \medskip

  To prove Property~(\ref{cor:analyticset}.\ref{il:hardy}), let $\gamma$ be
  any arc that satisfies the conditions of
  (\ref{cor:analyticset}.\ref{il:hardy}) and let $F$ be the associated
  deformation. For $t \in \Delta$, let
  $$
  \sigma_{F, t} \in H^0 \bigl( Y,\, (F_t)^*(T_X)\bigr)
  $$
  be the velocity vector field, as introduced in Notation~\ref{not:infdef} on
  page \pageref{not:infdef}. We aim to show that the $\sigma_{F, t}$ are
  really sections in $(F_t)^*(\sF)$.  Again, if any two numbers $r$, $s$ and
  any connected component $Y' \subseteq f^{-1}(Z_r^s)$ are given, it is clear
  from \eqref{eq:valentin} and \eqref{eq:ssp} that
  $$
  \sigma_{F, t}|_{U} \in H^0 \bigl( U,\, (F_t)^*(\sF_{Z_r^s})\bigr),
  $$
  where $\sF_{Z_r^s}$ is the sheaf introduced in Corollary~\ref{cor:newfrob}.
  Since $U$ is analytically open in the irreducible space $Y' \subseteq
  f^{-1}(Z_r^s)$ and since we have seen in Corollary~\ref{cor:newfrob} that
  $\sF_{Z_r^s}$ is a vector bundle, it follows immediately from the identity
  principle that
  $$
  \sigma_{F, t}|_{Y'} \in H^0 \bigl( Y' ,\, (F_t)^*(\sF_{Z_r^s})\bigr).
  $$
  Since this holds for all numbers $r$ and $s$, and all irreducible components
  $Y' \subseteq f^{-1}(Z_r^s)$,
  Property~(\ref{cor:analyticset}.\ref{il:hardy}) follows.
\end{proof}

\section{Proof of Theorem~\ref*{thm:mainresult}}
\label{sec:proof}

\subsection{Setup of notation, overview of the proof}

We end this paper with the proof of Theorem~\ref{thm:mainresult}. Throughout
the present Section~\ref{sec:proof}, we maintain the assumptions and the
notation of the theorem. In particular, we assume that we are given a morphism
$f :Y \to X$ of complex manifolds, with $Y$ compact, an involutive subsheaf
$\sF \subseteq T_X$ and a first order infinitesimal deformation of $f$,
denoted $\sigma \in H^0 \bigl( Y,\, \sF_Y \bigr)$, where $\sF_Y \subseteq
f^*(T_X)$ is the image of $f^*(\sF)$ under the pull-back of the inclusion
map. We also assume that $H^1\bigl( Y,\, \sF_Y \bigr) = \{0\}$.

The proof is given in three steps. Replacing the target manifold $X$ with the
product $Y \times X$, and the morphism $f$ with the natural graph map, we
first show that it suffices to prove Theorem~\ref{thm:mainresult} in the case
where $f$ is a closed immersion. In Step 2, we construct a setting where the
tangent vectors $\sigma(y) \in T_X|_y$ and the vector spaces $T_Y|_y \subseteq
T_X|_y$ are transversal at all points $y \in Y$. A third step will then
complete the proof.

\subsection{Step 1: Reduction to the case of a closed immersion}
\label{subsec:immersion}

In Section~\ref{ssec:outline} we have discussed the situation where $f$ is
a closed immersion, and where the infinitesimal deformation $\sigma$ was
locally given by restrictions of vector fields that live on open subsets of
$X$. In order to reduce to this simpler situation, we will show that to give a
deformation of $f$, it is equivalent to give a relative deformation of the
graph morphism,
$$
\iota: Y \to Y \times X \text{\quad where \quad} \iota(y) = \bigl(y, f(y)
\bigr),
$$
which is a closed immersion that identifies the domain $Y$ with the graph of
$f$. We will then aim to construct an involutive subsheaf $\sG \subseteq
T_{Y\times X}$ that comes from $\sF$, and an infinitesimal deformation of the
graph morphism $\iota$ along the sheaf $\sG$ that is related to $\sigma$.

For this, recall that the tangent bundle of the product is a direct sum
$T_{Y\times X} = \pi_Y^*(T_Y) \oplus \pi_X^*(T_X)$, where the $\pi_{\bullet}$
are the natural projections, and set $\sG := \{0\} \oplus \pi_X^*(\sF)
\subseteq T_{Y\times X}$. Since $\sG$ is generated by vector fields that are
$\pi_X$-related to vector fields in $\sF$, it follows from
\cite[Prop.~1.55]{Warner83} that $\sG$ is closed under the Lie bracket.
Finally, consider the first order infinitesimal deformation $\sigma_\iota:
Y\to \iota^*(T_{Y\times X})$ of $\iota$, given by $\sigma_\iota (y) := (0,
\sigma(y))$. 

The following lemmas are then immediate from the construction.

\begin{lem}\label{lem:kwieck}
  The infinitesimal deformation $\sigma_\iota$ is contained in the subspace
  $H^0 \bigl( Y, \Image (\iota^*(\sG) \to \iota^*(T_{X\times Y})) \bigr)
  \subseteq H^0 \bigl( Y, \iota^*(T_{X\times Y}) \bigr)$. \qed
\end{lem}

\begin{lem}\label{lem:brot}
  There exist natural isomorphisms $\iota^*(\sG) \cong f^*(\sF)$ and 
  $$
  \Image \bigl(\iota^*(\sG) \to \iota^*(T_{X\times Y})\bigr) \cong \Image
  \bigl( f^*(\sF) \to f^*(T_X)\bigr) =: \sF_Y.
  $$
  In particular, we have
  $H^1\bigl(Y,\,\Image(\iota^*(\sG)\to\iota^*(T_{X\times Y}))\bigr) =
  H^1\bigl(Y,\, \sF_Y \bigr) = \{0\}$. \qed
\end{lem}

\begin{lem}\label{lem:defgraph}
  If $F_\iota: \Delta\times Y\to Y\times X$ is a deformation of the graph
  morphism $\iota$ along $\sG$, then $F := \pi_X\circ F_\iota : \Delta\times Y
  \to X$ is a deformation of $f$ along $\sF$. If $F_\iota$ is a lifting of
  $\sigma_\iota$, then $F$ is a lifting of $\sigma$. \qed
\end{lem}

In summary, Lemmas~\ref{lem:kwieck}--\ref{lem:defgraph} show that all
assumptions made in Theorem~\ref{thm:mainresult} also hold for the morphism
$\iota$, and that it suffices to find a lifting of $\sigma_\iota$ along
$\sG$. Without loss of generality, we can therefore maintain the following
assumption throughout the rest of the proof.

\begin{assumption}
  The morphism $f : Y\to X$ is a closed immersion.
\end{assumption}

\subsection{Step 2: Time dependent vector fields} 

The explicit computations of {\v C}ech cocycles that we will use in Step~3 of
this proof become rather complicated if the infinitesimal deformation $\sigma$
has zeros or if its associated tangent vectors are not transversal to $f(Y)
\subseteq X$. As in Section~\ref{subsec:immersion}, we avoid this problem by
enlarging $X$.

\begin{construction}\label{cons:eta}
  Set $Z:=X\times \C$, with projections $\pi_X : Z \to X$ and $\pi_\C : Z \to
  \C$. Throughout the remainder of the proof, the coordinate on $\C$ will be
  denoted by $t$ and referred to as ``time''. Using that the tangent bundle of
  $Z$ decomposes as a direct sum, we consider the sheaf
  $$
  \sG := \pi_X^*(\sF)\oplus\pi_\C^*(T_\C) \,\subseteq\, \pi_X^*(T_X) \oplus
  \pi_\C^*(T_\C) = T_Z,
  $$
  the inclusion map
  $$
  g : Y \to Z, \quad y \mapsto \bigl( f(y), 0 \bigr).
  $$
  and the infinitesimal deformation
  $$
  \eta \in H^0\bigl( Y,\, g^*(T_Z) \bigr), \quad \eta := \sigma + \frac{d}{dt}.
  $$
\end{construction}

As in Section~\ref{subsec:immersion}, the following is immediate from the
construction.

\begin{lem}\label{lem:defZ}
  The sheaf $\sG$ is closed under Lie bracket. If $G: \Delta\times Y\to Z$ is
  a deformation of the morphism $g$ along $\sG$, then $F := \pi_X \circ G :
  \Delta\times Y \to X$ is a deformation of $f$ along $\sF$. If the
  deformation $G$ is a lifting of $\eta$, then $F$ is a lifting of
  $\sigma$. \qed
\end{lem}

\begin{warning}
  If $\sG_Y \subseteq g^*(T_Z)$ denotes the image of $g^*(\sG)$ under the
  pull-back of the inclusion map, then $\sG_Y = \sF_Y \oplus \sO_Y$. It is
  therefore generally wrong that $H^1\bigl( Y,\, \sG_Y \bigr) = \{0\}$, and
  the assumptions of Theorem~\ref{thm:mainresult} will generally not hold for
  the morphism $g$. Rather than using cohomological vanishing for $\sG_Y$, the
  arguments given in Step 3 will therefore only use cohomological vanishing of
  $\sF_Y$ and the special form of $\eta$, in order to construct infinitesimal
  liftings of arbitrary order.
\end{warning}

The following special types of vector fields on $Z$ will play a role in the
computations.

\begin{defn}[Time dependent vector field]
  A vector field on $Z$ is called a \emph{time dependent vector field in
    $\sF$} if it is a section of the sheaf
  $$
  \pi_X^*(\sF) \oplus \{0 \} \subseteq \pi_X^*(\sF) \oplus \pi_\C^*(T_\C)
  \subseteq T_Z.
  $$
\end{defn}

\begin{defn}[Vector field with constant flow in time]
  A vector field $D$ on $Z$ is called a \emph{vector field in $\sG$ with
    constant flow in time} if it is of the form
  $$
  D = D' + \frac{d}{dt},
  $$
  where $D'$ is a time dependent vector field in $\sF$.
\end{defn}

We remark that the first-order infinitesimal deformation $\eta$ of
Construction~\ref{cons:eta} is induced by a vector field with constant flow in
time, in the sense of the following definition.

\begin{defn}[Infinitesimal deformations induced by vector fields]\label{def:idiss}
  An $n$th order infinitesimal deformation $\eta_n : Y \to g^* \Jet^n(Z) =
  \Jet^n(Z)|_Y$ of the closed immersion $g$ is \emph{induced by vector fields
    in $\sG$ with constant flow in time} if for every point $y \in Y$ there is
  a neighborhood $U = U\bigl(g(y)\bigr) \subseteq Z$ and a vector field $D \in
  H^0 \bigl( U,\, \sG \bigr)$ with constant flow in time, such that the
  restriction $\eta_n|_{U\cap Y}$ is given by the section $\tau^n_D|_{U\cap
    Y}$ discussed in Definition~\ref{def:jetvec}.
\end{defn}

In step 3 of the proof, we need to consider iterated Lie brackets of vector
fields with constant flow in time. We end this section with an elementary
observation, asserting that Lie brackets of time dependent vector fields, or
of vector fields with constant flow in time will always be time dependent.

\begin{lem}\label{lem:lietime1}
  Let $U \subseteq Z$ be any open set and let $D_1$ and $D_2$ be any two time
  dependent vector fields in $\sF$, defined on $U$. Then the following holds.
  \begin{enumerate}
  \item\label{il:camillo} The Lie bracket $[D_1, D_2]$ is a time dependent
    vector field in $\sF$.
  \item\label{il:peppone} The Lie bracket $[\frac{d}{dt}, D_1]$ is a time
    dependent vector field in $\sF$.
  \end{enumerate}
\end{lem}
\begin{proof}
  Assertion~(\ref{lem:lietime1}.\ref{il:camillo}) follows from an elementary
  computation, cf.~\cite[Prop.~1.55]{Warner83}, when one observes that a
  vector field in $\sG$ is a time dependent vector field in $\sF$ if and only
  if it is $\pi_\C$-related to the trivial vector field $0 \in H^0 \bigl(
  \C,\, T_\C \bigr)$. Observing that a vector field has constant flow in time
  if and only if it is $\pi_\C$-related to the vector field $\frac{d}{dt} \in
  H^0 \bigl( \C,\, T_\C \bigr)$, the same computation also gives
  (\ref{lem:lietime1}.\ref{il:peppone}).
\end{proof}

\begin{cor}\label{cor:lietime}
  Let $D_1 + \frac{d}{dt}$ and $D_2+\frac{d}{dt}$ be any two vector fields in
  $\sG$ with constant flow in time. If $n$ is any number, then the iterated
  Lie bracket $[D_1 +\frac{d}{dt} , D_2+\frac{d}{dt} ]^{(n)}$ is a time
  dependent vector field in $\sF$. \qed
\end{cor}

\subsection{Step 3: End of proof} 

The end of the proof of Theorem~\ref{thm:mainresult} is now very similar to
the proof of Theorem~\ref{thm:class}. First, we prove an analogue of
Lemma~\ref{lem:liftnnpo} that gives liftings of $\eta$ to arbitrary
order. These liftings will locally be induced by vector fields in $\sF$ with
constant flow in time. Finally, we apply Artin's result to construct the
required deformation of $f$. The universal properties of the space
$\Hom_{\sF}\bigl(Y,\, X \bigr)$, as spelled out in
Corollary~\ref{cor:analyticset}, will then guarantee that this is in fact a
deformation along the subsheaf $\sF$.

\begin{lem}\label{lem:liftalongf}
  Let $\eta_n : Y\to g^* \Jet^n(Z)$ be an $n$th order infinitesimal
  deformation of the closed immersion $g$ that is induced by vector fields in
  $\sG$ with constant flow in time. Then there exists a lifting $\eta_{n+1} :
  Y \to g^* \Jet^{n+1}(Z)$ of $\eta_n$ that is likewise induced by vector
  fields in $\sG$ with constant flow in time.
\end{lem}

\begin{proof}
  As a first step, we construct liftings locally. Using the cohomological
  vanishing for $\sF_Y$, we can then correct the local liftings, to ensure
  that they glue on overlaps. This will define a global lifting, which is then
  shown to be induced by vector fields in $\sG$ with constant flow in time.

  It follows from Definition~\ref{def:idiss} that there exists an acyclic
  covering of $g(Y) \subseteq Z$ by open subsets $(U_i)_{i \in I} \subseteq Z$
  such that there are time dependent vector fields $D_i \in H^0 \bigl(U_i,\,
  \pi_X^*(\sF) \oplus \{0\} \bigr)$ that satisfy $\eta_n|_{U_i\cap Y} =
  \tau^n_{D_i+\frac{d}{dt}}|_{U_i\cap Y}$. We consider the induced section of
  the $(n+1)$st jet bundle,
  $$
  \tau_i := \tau^{n+1}_{D_i+\frac{d}{dt}}\bigr|_{U_i\cap Y} : U_i\cap Y \to
  \Jet^{n+1}(Z).
  $$
  Obviously, the $\tau_i$ are local liftings of $\eta_n$, but they do not
  necessarily glue on overlaps. However, it follows from Theorem~\ref{thm:dj}
  that for any pair of indices $i$, $j \in I$, the affine differences are
  given by iterated Lie brackets,
  \begin{align*}
    \nu_{i,j} & := \tau_i\bigr|_{U_i\cap U_j\cap Y} - \tau_j\bigr|_{U_i\cap
      U_j\cap Y} \\ & \,\,\, = \underbrace{\left[D_i+ \frac{d}{dt}
        \Bigr|_{U_i\cap U_j} ,\,\, D_j+ \frac{d}{dt} \Bigr|_{U_i\cap
          U_j}\right ]^{(n+1)}}_{=: A_{i,j}} \Biggr|_{U_i\cap U_j\cap Y.}
  \end{align*}
  Corollary~\ref{cor:lietime} asserts that the iterated Lie brackets $A_{i,j}$
  are time-dependent vector fields in $\sF$. The differences $\nu_{i,j}$
  therefore yield cohomology classes in $H^1\bigl(Y,\, \sF_Y \bigr)$ which are
  zero by assumption. We can thus find sections $\lambda_i \in H^0 \bigl( U_i
  \cap Y,\, \sF_Y \bigr)$ such that $\lambda_i - \lambda_j = \nu_{i,j}$. As in
  the proof of Lemma~\ref{lem:liftnnpo}, viewing the $\lambda_i$ as sections
  in $H^0 \bigl( U_i \cap Y,\, \sF_Y \oplus \sO_Y \bigr) = H^0 \bigl( U_i \cap
  Y,\, \sG_Y \bigr) \subseteq H^0 \bigl( U_i \cap Y,\, g^*(T_Z) \bigr)$, the
  sections obtained by translation,
  $$
  \tau_i - \lambda_i : U_i\cap Y \to \Jet^{n+1}(Z)|_{U_i \cap Y},
  $$
  glue on overlaps $U_i\cap U_j\cap Y$ and define a lifting to $n+1$st order,
  \begin{equation}\label{eq:archimedes}
    \eta_{n+1}: Y \to \Jet^{n+1}(Z), \text{\quad with \quad} \eta_{n+1}|_{
      U_i \cap Y} = \tau_i - \lambda_i \text{ for all $i$.}
  \end{equation}

  It remains to show that $\eta_{n+1}$ is an infinitesimal deformation induced
  by vector fields in $\sG$ with constant flow in time. To check the
  conditions of Definition~\ref{def:idiss}, let $y \in Y$ be any point, and
  let $i \in I$ be any index with $g(y) \in U_i$. Then it suffices to
  construct a time dependent vector field $D \in H^0 \bigl(U_i,\, \pi_X^*(\sF)
  \bigr)$ such that $\eta_{n+1}|_{U_i \cap Y} =
  \tau^{n+1}_{D+\frac{d}{dt}}|_{U_i \cap Y}$.

  To this end, consider the sections $\tau_i$ and $\lambda_i$ defined
  above. Recall that $\tau_i$ is induced by the vector field
  $D_i+\frac{d}{dt}$. Since covering of $Z$ is acyclic, the section $\lambda_i
  \in H^0 \bigl(U_i \cap Y,\, \sF_Y \bigr)$ is given as the restriction of a
  vector field $E \in H^0 \bigl( U_i,\, \pi_X^*(\sF) \bigr)$.  Set $D := (D_i
  - \frac{t^n}{n!}E)$. With Theorem~\ref{thm:dj} at hand, it is then easy to
  compute the affine differences of $\tau^{n+1}_{D+\frac{d}{dt}}|_{U_i \cap
    Y}$ and $\tau^{n+1}_{D_i+\frac{d}{dt}}|_{U_i \cap Y}$ on $U_i$ as
  $$
  \tau^{n+1}_{D+\frac{d}{dt}}|_{U_i \cap Y} -
  \underbrace{\tau^{n+1}_{D_i+\frac{d}{dt}}|_{U_i \cap Y}}_{= \tau_i} = -E =
  -\lambda_i.
  $$
  We obtain $\tau^{n+1}_{D+\frac{d}{dt}}|_{U_i \cap Y} = \tau_i - \lambda_i$
  and Equation~\eqref{eq:archimedes} then gives
  $\tau^{n+1}_{D+\frac{d}{dt}}|_{U_i \cap Y} = \eta_{n+1}|_{U_i \cap Y}$, as
  required.
\end{proof}

\begin{proof}[Proof of Theorem~\ref{thm:mainresult}, end of proof]
  Consider the analytic subset $\Hom_\sG(Y,Z)$ of the Douady space $\Hom(Y,Z)$
  constructed in Corollary~\ref{cor:analyticset} and the sequence of liftings
  $\eta_1$, $\eta_2$, $\ldots$ of Lemma~\ref{lem:liftalongf}. By
  Proposition~\ref{prop:infdefn}, we can view the $\eta_i$ as morphisms $\Spec
  \mathbb C[\varepsilon]/(\varepsilon^{i+1}) \times Y \to Z$.
  Assertion~(\ref{cor:analyticset}.\ref{il:laurel}) of
  Corollary~\ref{cor:analyticset} then implies that these morphisms factor via
  $\Hom_\sG(Y,Z)$, for each $i$.

  Arguing as in the proof of Theorem~\ref{thm:class}, only replacing
  $\Hom(Y,Z)$ by the analytic subspace $\Hom_\sG(Y,Z)$, Artin's Theorem
  \cite[Thm.~1.2]{ARTIN68} guarantees the existence of a deformation $G$ of
  $g$ that factors via $\Hom_\sG(Y,Z)$ and lifts the infinitesimal deformation
  $\eta$. Lemma~\ref{lem:defZ} and
  Assertion~(\ref{cor:analyticset}.\ref{il:hardy}) of
  Corollary~\ref{cor:analyticset} then implies that $F=\pi_X\circ G$ is in
  fact a deformation along $\sF$ that lifts the infinitesimal deformation
  $\sigma$.
\end{proof}

\providecommand{\bysame}{\leavevmode\hbox to3em{\hrulefill}\thinspace}
\providecommand{\MR}{\relax\ifhmode\unskip\space\fi MR}
\providecommand{\MRhref}[2]{%
  \href{http://www.ams.org/mathscinet-getitem?mr=#1}{#2}
}
\providecommand{\href}[2]{#2}

\end{document}